\documentclass[]{article}
\usepackage{amsmath,amssymb,amsthm,mathtools}
\usepackage{graphicx}
\usepackage{verbatim}
\usepackage[mathscr]{euscript}
\usepackage[T1]{fontenc}

\theoremstyle{plain}
\newtheorem{thm}{Theorem}

\newtheorem{cor}{Corollary}
\newtheorem{prop}{Proposition}

\theoremstyle{definition}
\newtheorem{defn}{Definition}

\title{Recovering functions from the modulation space $\mathscr{F}W$}

\author{Jeff Ledford}
\date{}

\begin{document}
\allowdisplaybreaks

\maketitle

\begin{abstract}
In this short note we show that functions in the modulation space $\mathscr{F}W=\{ f: \sum_{j\in\mathbb{Z}^n}\| \hat{f}(\cdot+2\pi j)\|_{L_\infty([-\pi,\pi]^n)}<\infty \}$ enjoy similar recovery properties as band-limited functions.  If $\{\phi_\alpha\}$ is a regular family of cardinal interpolators, then one can build an approximand of $f$ using the fundamental functions corresponding to $\phi_\alpha$.  Then taking the appropriate limit, one recovers $f$ both in norm and pointwise. 
\end{abstract}


\section{Introduction}

This note continues the study of approximation and recovery of functions using cardinal interpolants.  This problem has a rich history which has tended to focus on a specific approximation scheme, e.g. spline, multiquadric.  A few results in this area may be found in \cite{MN}, \cite{RS}, and \cite{baxter}.  Each deals with the recovery of band-limited functions.

It was the goal of \cite{z paper} to unite these methods under a single framework, regular families of cardinal interpolators.  The goal of this work is to use the same framework and move beyond band-limited functions.  A similar study has been done for scattered data in \cite{me 2}, although in a limited context.  The main result there is analogous to Theorem 2 below.  Similar work has also been carried out by Hamm in \cite{H}, where Sobolev functions are recovered.

The remainder of this note is organized as follows.  Section 2 contains definitions and basic facts needed to set up our problem.  The main result is stated and proved in Section 3, while the final section contains examples of regular families of cardinal intepolators.


\section{Definitions and Basic Facts}

We begin with a convention for the Fourier transform.
\begin{defn}
The \emph{Fourier transform} of a function $f(x)\in L^1(\mathbb{R}^n)$, is defined to be the function
\[
\hat{f}(\xi):=(2\pi)^{-n/2}\int_{\mathbb{R}^n}f(x)e^{-i\langle x,\xi \rangle }dx.
\]
\end{defn}
We make the usual extension to the class of tempered distributions using this convention.

We will be interested in the following modulation spaces.  Let $1\leq p \leq \infty$ and define
\begin{align*}
W(L_p,\ell_1):=&\left\{ f\in L_\infty(\mathbb{R}): \sum_{j\in\mathbb{Z}^n}\| f(\cdot+2\pi j) \|_{L_p([-\pi,\pi]^n)} <\infty  \right\}, \text{ and}\\
\mathscr{F}W(L_p,\ell_1):=&\left\{ f: \hat{f}\in W(L_p,\ell_1)   \right\}
\end{align*}
The space $W(L_\infty, \ell_1)$ is called the Wiener space.  These spaces are sometimes called Wiener amalgam spaces. To avoid cumbersome notation we will abbreviate these spaces as $W_p$ and $\mathscr{F}W_p$, respectively.  Additionally, we will write $W:=W(L_\infty, \ell_1)$ for the Wiener space.  For more information on modulation spaces the reader may consult Capters 6 and 11 of \cite{G}.
We will also make use of the classical Paley-Wiener space $PW$, given by
\[
PW:=\{f\in L_2(\mathbb{R}): \hat f (\xi)=0 \text{ a.e } \xi \notin [-\pi,\pi]^n  \}.
\]
The norms on these spaces are given by
\begin{align*}
\|f \|_{W_p}&:=\sum_{j\in\mathbb{Z}^n}\| f(\cdot+2\pi j) \|_{L_p([-\pi,\pi]^n)}, \\
\| f \|_{\mathscr{F}W_p}&:=\sum_{j\in\mathbb{Z}^n}\| \hat f(\cdot+2\pi j) \|_{L_p([-\pi,\pi]^n)}, \text{ and}\\
\|  f \|_{PW}&:=\| \hat{f} \|_{L_2([-\pi,\pi]^n)}
\end{align*}
Note the straightforward inclusion $\mathscr{F}W_p\subset \mathscr{F}L_p(\mathbb{R}^n)$ and $\mathscr{F}W\subset\mathscr{F}W_p$ for $1\leq p\leq \infty$.  Our next definitions come from \cite{z paper}.

\begin{defn}  We say that a function $\phi$ is a \emph{cardinal interpolator} if it satisfies the following conditions:
\begin{enumerate}
\item[(H1)]$\phi$ is a real valued slowly increasing function on $\mathbb{R}^n$,
\item[(H2)] $\hat{{\phi}}(\xi) \geq 0$ and $\hat{{\phi}}(\xi)\geq \delta >0$ in $[-\pi,\pi]^n$,
\item[(H3)] $\hat{{\phi}}\in C^{n+1}(\mathbb{R}^n\setminus\{0\})$,
\item[(H4)] There exists $\epsilon >0$ such that if $|\alpha|\leq n+1$, \\ $D^{\alpha}\hat\phi(\xi)=O(\| \xi \|^{-(n+\epsilon)})$ as $\|\xi\|\to\infty$,
\item[(H5)] For any multi-index $\alpha$, with $|\alpha|\leq n+1$, 
\[
\dfrac{\displaystyle\prod_{j=1}^{|\alpha|}D^{\alpha_j}\hat\phi  }     {[\hat\phi]^{|\alpha|+1}} \in L^{\infty}([-\pi,\pi]^n) \hspace{.25 in}\text{where}\hspace{.25in} \sum_{j=1}^{|\alpha|}\alpha_j = \alpha.
\] 
\end{enumerate}
\end{defn}

\begin{defn}
We call a family of functions $\{\phi_{\alpha}\}_{\alpha\in A}$ a \emph{regular family of cardinal interpolators} if for each $\alpha\in A\subset(0,\infty)$, $\phi_{\alpha}$ is a cardinal interpolator and in addition to this, we have:
\begin{enumerate}
\item[(R1)] For $j\in\mathbb{Z}^n\setminus\{0\}$, $\xi\in[-\pi,\pi]^n$, define $M_{j,\alpha}(\xi)=\displaystyle\dfrac{\hat{\phi}_\alpha(\xi+2\pi j)}{\hat{\phi}_\alpha(\xi)} $, then for $j\in\mathbb{Z}^n\setminus\{0\}$, $\displaystyle\lim_{\alpha\to\infty}M_{j,\alpha}(\xi)=0$ for almost every $\xi\in[-\pi,\pi]^n$.
\item[(R2)] There exists $M_j\in l^1(\mathbb{Z}^n\setminus\{0\})$, independent of $\alpha$, such that for all $j\in\mathbb{Z}^n\setminus\{0\}$, $M_{j,\alpha}(\xi)\leq M_j$.
\end{enumerate}
\end{defn}
These definitions allow us to form a fundamental function $L_\alpha$ associated to $\phi_\alpha$, defined by its Fourier transform
\[
\hat{L}_{\alpha}(\xi)=(2\pi)^{-n/2}\dfrac{\hat\phi_{\alpha}(\xi)}{\sum_{j\in\mathbb{Z}^n}\hat\phi_{\alpha}(\xi+2\pi j)}.
\]  
We will need a few properties which may be found in \cite{z paper}.
\begin{prop}[Proposition 3, \cite{z paper}]\label{Lto0}
If $\{\phi_\alpha\}$ is regular a regular family of cardinal interpolators, then $\displaystyle \lim_{\alpha\to\infty}\sum_{j\neq0}\hat{L}_{\phi_\alpha}(\xi+2\pi j)=0$ for a.e. $\xi\in[-\pi,\pi]^n$ and in $L^1([-\pi,\pi]^n)$.
\end{prop}
\begin{prop}[Theorem 1, \cite{z paper}]
Suppose that $\{\phi_{\alpha}\}$ is a regular family of cardinal interpolators and let $f\in PW_\pi$.  Then we have the following limits:
\begin{enumerate}
\item[$(a)$] $\displaystyle \lim_{\alpha\to\infty}\bigg\| \sum_{j\in\mathbb{Z}^n}f(j)L_{{\phi}_\alpha}(\cdot-j) - f \bigg \|_{L^2(\mathbb{R}^n)} =0,$
\item[$(b)$] $\displaystyle \lim_{\alpha\to\infty}\bigg| \sum_{j\in\mathbb{Z}^n}f(j)L_{{\phi}_\alpha}(\cdot-j) - f \bigg| =0$  uniformly in $\mathbb{R}^n$.
\end{enumerate}
\end{prop}

Our goal is to build approximands of a given function $f$ with sufficiently nice properties.  To do this we first need $\{f_k:k\in\mathbb{Z}^n \}$.  We define $f_k$ by its Fourier transform
\begin{equation}
\hat{f}_k(\xi)=\hat{f}(\xi+2\pi k)\chi_{[-\pi,\pi]^n}(\xi).
\end{equation}
Now our approximand will take the form
\begin{equation}
J_{\alpha}[f](x):= \sum_{j,k\in\mathbb{Z}^n}f_{k}(j)L_{\alpha}(x-j)e^{2\pi i \langle x,k \rangle }.
\end{equation}

\section{Main Result}

Our main result is the following theorem.
\begin{thm}
Suppose that $f\in \mathscr{F}W$ and $\{\phi_\alpha \}$ a regular family of cardinal interpolators. We have
\[
\lim_{\alpha\to\infty}\| f- J_\alpha[f]  \|_{\mathscr{F}W_p}=0,
\]
for $1\leq p \leq \infty$.
\end{thm}

\begin{proof}
We note that $f\in \mathscr{F}W$ implies that $f\in \mathcal{F}W_p$ for all $1\leq p\leq \infty$, additionally $\{f_k:k\in\mathbb{Z}^n\}\subset PW$.  As a result, we may use the formula
\begin{equation}\label{hatf}
(2\pi)^{n/2}\hat{f_k}(\xi)=\sum_{j\in\mathbb{Z}^n}f_k(j)e^{-i\langle j, \xi \rangle }, \quad \text{a.e. }\xi\in [-\pi,\pi]^n. 
\end{equation}
Now we may begin our calculation.
\begin{align*}
& \| f- J_\alpha[f] \|_{\mathscr{F}W_p}\\
=& \sum_{l\in\mathbb{Z}^n} \|  \hat{f}(\cdot +2\pi l) - \widehat{J_\alpha[f]}(\cdot+2\pi l)    \|_{L_p([-\pi,\pi]^n)}\\
=& \sum_{l\in\mathbb{Z}^n} \| \hat{f}(\cdot +2\pi l) - (2\pi)^{n/2}\sum_{k\in\mathbb{Z}^n}\hat{L}_\alpha(\cdot+2\pi (l-k))\hat{f}_{k}     \|_{L_p([-\pi,\pi]^n}\\
\leq & \sum_{l\in\mathbb{Z}^n} \| (1-(2\pi)^{n/2}\hat{L}_{\alpha})  \hat{f}_l   \|_{L_p([-\pi,\pi]^n} \\
&\quad+ (2\pi)^{n/2}\sum_{l\in\mathbb{Z}^n}\| \sum_{k\neq l} \hat{L}_\alpha(\cdot+2\pi (l-k))\hat{f}_{k} \|_{L_p([\pi,\pi]^n)}\\
=&  (2\pi)^{n/2}\sum_{l\in\mathbb{Z}^n} \| \sum_{k\neq 0} \hat{L}_{\alpha}(\cdot+2\pi k)  \hat{f}_l   \|_{L_p([-\pi,\pi]^n} \\
&\quad+(2\pi)^{n/2}\sum_{l\in\mathbb{Z}^n}\| \sum_{k\neq 0} \hat{L}_\alpha(\cdot+2\pi k)\hat{f}_{l-k} \|_{L_p([\pi,\pi]^n)}\\
\leq & 2(2\pi)^{n/2}\sum_{k\neq 0}M_k\sum_{l\in\mathbb{Z}^n}\| \hat{f}_l \|_{L_p([-\pi,\pi]^n)} \\
=& \left(2(2\pi)^{n/2}\sum_{k\neq 0}M_k\right) \| f \|_{\mathscr{F}W_p}
\end{align*}

The first two equalities are just the definition and the formula \eqref{hatf}.  The first inequality is the triangle inequality.  The next equality follows from the fact that
\[
 1-\hat{L}_\alpha(\xi)= (2\pi)^{n/2}\sum_{k\neq 0}\hat{L}_\alpha(\xi+2\pi k).
 \] 
The final inequality is $(R2)$, where we have switched the sums with Tonelli's theorem.
Now Proposition \ref{Lto0} and the Dominated Convergence theorem finish the proof.

\end{proof}

This theorem leads to a pointwise result.

\begin{cor}
Under the hypotheses of Theorem 1, we have
\[
\lim_{\alpha\to 0}|f(x)-J_\alpha[f](x)|=0
\]
uniformly on $\mathbb{R}^n$.
\end{cor}
\begin{proof}
To see this, we use the Fourier integral representation.
\begin{align*}
&|f(x)-J_\alpha[f](x)|\\
\leq & (2\pi)^{-n/2}\int_{\mathbb{R}^n}|\hat{f}(\xi)-\widehat{J_\alpha[f]}(\xi)|d\xi\\
=&(2\pi)^{-n/2}\sum_{l\in\mathbb{Z}^n} \| \hat{f}(\cdot+2\pi l)-\widehat{J_\alpha[f]}(\cdot+2\pi l)   \|_{L_1([-\pi,\pi]^n)}\\
\leq& (2\pi)^{1-1/p-n/2}\sum_{l\in\mathbb{Z}^n}\| \hat{f}(\cdot+2\pi l)-\widehat{J_\alpha[f]}(\cdot+2\pi l) \|_{L_p([-\pi,\pi]^n)}\\
=& (2\pi)^{1-1/p-n/2} \| f-J_\alpha[f] \|_{\mathscr{F}W_p}
\end{align*}
Here we have periodized the first integral, then applied H\"{o}lder's inequality.  Now taking the limit yields the result.
\end{proof}
We also get an $L_p$ result of sorts.
\begin{cor}
Under the hypotheses of Theorem 1, we have
\[
\lim_{\alpha\to 0}\|f-J_\alpha[f]\|_{\mathscr{F}L_p(\mathbb{R}^n)}=0.
\]
\end{cor}
\begin{proof}
To see this we need only note that
\[
\|  f \|_{\mathscr{F}(\mathbb{R}^n)} = \| \hat{f} \|_{L_p(\mathbb{R}^n)} \leq \sum_{k\in\mathbb{Z}^n}\| \hat{f}(\cdot-2\pi k)  \|_{L_p([-\pi,\pi]^n)}=\| f \|_{\mathscr{F}W_p}.
\]
\end{proof}

For certain $p$, we can relax the conditions on $f$ somewhat.  Specifically, we have the following.
\begin{thm}
Suppose that $f\in \mathscr{F}W_2$ and $\{\phi_\alpha\}$ is a regular family of cardinal interpolators, then we have
\begin{enumerate}
\item[(a)] $ \displaystyle\lim_{\alpha\to\infty}\| f-J_\alpha[f]  \|_{\mathscr{F}W_2}=0$,
\item[(b)] $ \displaystyle\lim_{\alpha\to\infty}\| f-J_\alpha[f]  \|_{L_2(\mathbb{R}^n)}=0$, and
\item[(c)] $ \displaystyle\lim_{\alpha\to\infty}| f(x)-J_\alpha[f](x)|=0,$ uniformly on $\mathbb{R}^n$.
\end{enumerate}
\end{thm}
\begin{proof}
The proof is the same \emph{mutatis mutandis}.  We simply note that in this case we can use \eqref{hatf} and carry out the calculation as before.
\end{proof}

\section{Examples}
In this section we provide concrete examples of families of functions which exhibit this behavior.  Since any family of regular cardinal interpolators will work, we can use the examples found in \cite{z paper}.  For more information regarding these functions the reader may consult \cite{z paper} and the references found there.
\subsection{Polyharmonic Cardinal Splines}
Suppose that $\Delta$ is the $n$-dimensional Laplacian and $\Delta^k f=\Delta(\Delta^{k-1}f)$.  For $k\geq 1$, if $\Delta^k f =0$ on $\mathbb{R}^n\setminus\mathbb{Z}^n$ and $f\in C^{2k-2}(\mathbb{R}^n)$, then $f$ is called a polyharmonic spline.  We note that in this case we have $\hat\phi_{k}(\xi)=\|\xi\|^{-2k}$.  The appropriate family of cardinal interpolators is given by $\{ \phi_k: k\in\mathbb{N}  \}$.
\subsection{Gaussians}
The family given by $\{ e^{-\|\cdot\|^2/(4\alpha)}: \alpha\geq 1   \}$ is a regular family of cardinal interpolators.
\subsection{Multiquadrics}
The family given by $\{ (\| \cdot \|^2+c^2)^\alpha_j: j\in\mathbb{N}, c>0 \text{ fixed}   \}$ where $\{\alpha_j\}\subset [1/2,\infty )$ and dist$(\{a_j\},\mathbb{N})>0$ is a regular family of cardinal interpolators.  This is also true if we consider $\{ (\|\cdot\|^2+c^2)^{\alpha}: c\geq 1   \}$ here we may take $\alpha\in \mathbb{R}\setminus\mathbb{N}_0$, this final result was shown in \cite{HL}, with a partial result appearing in \cite{z paper}.



\begin{thebibliography}{00}


\bibitem{baxter}
B. Baxter,
``The asymptotic cardinal function of the multiquadratic $\varphi(r)=(r^2+c^2)^{1/2}$ as $c\to\infty$,''
Comput. Math. Appl. 24 (1992), no. 12, 1-6

\bibitem{G}
K. Gr{\"o}chenig,
\emph{Foundations of Time Frequency Analysis},
Birkh{\"a}user, Boston, MA, 2001

\bibitem{H}
K. Hamm, ``Approximation Rates for Interpolation of Sobolev Functions via Gaussians and Allied Functions,''
J. Approx. Theory, 189, (2015), 101-122

\bibitem{HL}
K. Hamm \and J. Ledford,
``Cardinal interpolation with general multiquadrics,'' (preprint) {\tt arXiv:1501.01899} 


\bibitem{z paper}
J. Ledford,
``On the convergence of regular families of cardinal interpolators," 
Adv. Comp. Math. (to appear)

\bibitem{me 2}
``Recovering functions from the Paley-Wiener amalgam space,''
(preprint) {\tt arXiv:1311.5169}

\bibitem{MN}
W. Madych \and S. Nelson,
``Polyharmonic Cardinal Splines,''
J. Approx. Theory 60,  1990), no.2, 141-156

\bibitem{RS}
S. Riemenschneider and N. Sivakumar,
``On cardinal interpolation by Gaussian radial-basis functions: properties of fundamental functions and estimates for Lebesgue constants''
 J. Anal. Math. 79, (1999), 33-61.







\end{thebibliography}
\end{document}